\documentclass[11pt]{article}
\usepackage{color,amsmath, amsthm, amsfonts, amssymb, dsfont, graphicx, listings, graphics, epsfig, enumerate,bbm,dsfont,mathrsfs}

\usepackage[margin=1in]{geometry}
\usepackage[fleqn,tbtags]{mathtools}
\usepackage[english]{babel}

\usepackage[round]{natbib} 

\title{Set-valued convex compositions}
\author{\c{C}a\u{g}{\i}n Ararat\thanks{Bilkent University, Department of Industrial Engineering, Ankara, Turkey, cararat@bilkent.edu.tr.}
}
\date{June 27, 2023}

\addto\captionsenglish {}
\makeatletter \renewenvironment{proof}[1][\proofname] {\par\pushQED{\qed}\normalfont\topsep6\p@\@plus6\p@\relax\trivlist\item[\hskip\labelsep\bfseries#1\@addpunct{.}]\ignorespaces}{\popQED\endtrivlist\@endpefalse} \makeatother
\newtheorem{theorem}{Theorem}[section]
\newtheorem{corollary}[theorem]{Corollary}
\newtheorem{lemma}[theorem]{Lemma}
\newtheorem{proposition}[theorem]{Proposition}
\newtheorem{assumption}[theorem]{Assumption}

\theoremstyle{definition}

\newtheorem{remark}[theorem]{Remark}

\numberwithin{equation}{section}

\newcommand{\R}{\mathbb{R}}

\DeclareMathOperator{\cl}{cl}

\DeclareMathOperator{\co}{conv}

\DeclareMathOperator{\dom}{dom}
\DeclareMathOperator{\epi}{epi}
\DeclareMathOperator{\gr}{gr}

\newcommand{\W}{\mathcal{W}}

\newcommand{\G}{\mathscr{G}}
\newcommand{\F}{\mathscr{F}}

\newcommand{\K}{\mathcal{K}}
\newcommand{\X}{\mathcal{X}}
\newcommand{\Y}{\mathcal{Y}}
\newcommand{\Z}{\mathcal{Z}}
\newcommand{\A}{\mathscr{A}}

\renewcommand{\P}{\mathscr{P}}

\newcommand{\of}[1]{\ensuremath{\left( #1 \right)}}
\newcommand{\cb}[1]{\ensuremath{ \left\{ #1 \right\} }}

\newcommand{\ip}[1]{\ensuremath{ \left\langle #1 \right\rangle }}

\def\prehp(#1,#2){\ensuremath{  #1 \cdot #2 }}

\begin{document}
\maketitle
\thispagestyle{empty}

\begin{abstract}
We study the composition of two set-valued functions defined on locally convex topological linear spaces. We assume that these functions map into certain complete lattices of sets that have been used to establish a conjugation theory for set-valued functions in the literature. Our main result is a formula for the conjugate of the composition in terms of the conjugates of the ingredient functions. As a special case, when the composition is proper and has further regularity, our formula yields a dual representation for the composition. The proof of the main result uses Lagrange duality and minimax theory in a nontrivial way.\\
\\[-5pt]
\textbf{Keywords and phrases:} complete lattice, set-valued function, convex function, composition, conjugation, Fenchel-Moreau theorem\\
\\[-5pt]
\textbf{Mathematics Subject Classification (2020): }06B23, 26E25, 46A55, 49J53, 90C48.
\end{abstract}

\section{Introduction}\label{sec:intro}

Conjugation is one of the fundamental concepts in convex analysis. Given an extended real-valued function $\rho$ on a locally convex topological linear space $\X$, the conjugate function of $\rho$ is defined as a weak*-lower semicontinuous convex function on the dual space $\X^\ast$ of $\X$. Then, by switching the roles of the primal and dual spaces, the biconjugate of $\rho$ is defined as a lower semicontinuous convex function on $\X$. The famous Fenchel-Moreau theorem states that $\rho$ coincides with its biconjugate provided that $\rho$ be proper, convex, and lower semicontinuous. In this case, the theorem provides a dual representation for $\rho$ in terms of its conjugate function as a supremum over the elements of $\X^\ast$; hence, $\rho$ is equivalently described by its conjugate function.

In view of the Fenchel-Moreau theorem, when the proper, convex, lower semicontinuous function $\rho$ is defined in terms of several functions, calculating the conjugate of $\rho$ in terms of the constituent functions becomes an important task for expressing the dual representation of $\rho$. Such structures include sums, conic combinations, compositions, infimal convolutions, and so on. Conjugation formulae for these structures are available in the literature, some standard ones can be found in \citet[Chapter 2]{zalinescu}. We also refer the reader to the recent work \citet{mucahitthesis} for duality results for extended real-valued quasiconvex compositions.

From an application point of view, a special class of convex functions, called \emph{convex risk measures}, defined on Lebesgue spaces are frequently used in financial mathematics. These are monotone, translative, and convex functions that are used to calculate capital requirements for uncertain financial positions. In this setting, Fenchel-Moreau theorem applied to a convex risk measure yields a dual representation that can be interpreted as a worst-case risk evaluation under Knightian uncertainty (or model uncertainty). We refer the reader to \citet[Sections~4.2,~4.3]{fs:sf} for the interplay between conjugation theory and convex risk measures.

The focus of this paper is on set-valued functions rather than on extended real-valued functions. In this case, we replace $\rho$ with a function $R$ defined on $\X$ mapping into the power set of another locally convex topological linear space $\Z$, that is, $R(x)\subseteq \Z$ for each $x\in\X$. In set-valued analysis, using the entire power set of $\Z$ as the image space of $R$ generally makes the study of $R$ intractable. Hence, one restricts attention to a certain class of subsets of $\Z$ such as closed sets, closed convex sets, compact sets, convex compact sets, and so on.

In the literature, several attempts have been made to generalize the concepts and results of convex analysis to the set-valued setting. In this paper, we follow the approach based on \emph{complete lattices}; see \citet{setoptsurv} for a detailed survey. More precisely, one assumes that $\Z$ is endowed with a preorder that is compatible with the topological linear space structure and extends the preorder to the power set of $\Z$ by introducing some \emph{set relations}. These set relations can be used to partition the power set into equivalence classes, each of which is represented by a unique element of a certain class of subsets of $\Z$. In particular, the class of all representatives is a complete lattice, that is, every subset of it has an infimum and supremum in the order-theoretic sense. Consequently, one can restrict attention to functions that map into this complete lattice and operate with these functions in a similar way to extended real-valued functions.

Using complete lattices induced by set relations, a conjugation theory for set-valued functions is established in \citet{ham09,ham11}; in particular, a set-valued generalization of the Fenchel-Moreau theorem is proved in \citet{ham09}. Later, in \citet{completeduality}, a duality theory for set-valued quasiconvex functions is constructed. Parallel to these developments in set-valued convex analysis, convex risk measures have been generalized to the set-valued setting in \citet{hh:duality}. These so-called \emph{set-valued convex risk measures} are defined on Lebesgue spaces of random vectors, and they have found applications in markets with transaction costs (e.g., \citet{hh:duality,svdrm}) and systemic risk measures (e.g., \citet{frw,systriskduality}).

When the set-valued function is defined in terms of several set-valued functions, calculating the conjugate of $R$ in terms of these of the constituent functions is generally more complicated compared to the extended real-valued setting. The main reason is that, in the set-valued setting, these calculations typically involve an additional operation called \emph{scalarization}, which is defined through minimizing a continuous linear function over the realization of the set-valued function. Having said this, when $R$ is the sum or infimal convolution of two set-valued functions, obtaining conjugation formulae is relatively easy and such formulae have been obtained in \citet[Section~6.1]{ham11} and \citet[Section~4.4]{ham09}, respectively.

In this paper, we study the \emph{composition} of two set-valued functions $F, G$ mapping into a complete lattice. We first consider the basic properties of the composition such as convexity, closedness, and properness in terms of the analogous properties of $F, G$. Then, we tackle the more challenging problem of calculating the conjugate function of the composition in terms of the conjugate functions of $F, G$. The proof of the main result (Theorem~\ref{thm:conj}) relies on several technical observations together with the use of Liu's minimax inequality (see \citet{liu,liurelated1,liurelated2}), which works under weaker conditions than Sion's minimax equality (see \citet{sion}). We also use Lagrange duality to obtain the final version of the conjugation formula with a particular attention paid to the properness of the scalarizations of $F, G$. As a corollary of the main theorem, we provide a dual representation for a convex composition provided that it be proper and satisfy a semicontinuity condition.

The rest of the paper is organized as follows. In Section~\ref{sec:prelim}, we recall some basic concepts in convex and set-valued analysis. Section~\ref{sec:comp} is devoted to set-valued convex compositions, the subject matter of the paper. The proof of Theorem~\ref{thm:conj}, the main result, is presented separately in Section~\ref{sec:proof} with a technical preparation before the actual proof. We finish the paper with some concluding remarks in Section~\ref{sec:conc}.

\section{Preliminaries}\label{sec:prelim}

In this section, we review some preliminary notions and results in convex analysis for extended real-valued and set-valued functions. The book \citet{zalinescu} is a standard reference for classical convex analysis in infinite-dimensions. For the set-valued case, we refer the reader to the pioneering work \citet{ham09} and the survey article \citet{setoptsurv}.

\subsection{Extended real-valued functions}\label{sec:extended}

Let $\X$ be a Hausdorff locally convex topological real linear space with topological dual $\X^\ast$. We denote by $\ip{\cdot,\cdot}\colon\X^\ast\times\X\to\R$ the bilinear duality mapping between $\X^\ast$ and $\X$. Let us fix a neighborhood base $\mathcal{N}(\X)$ of $0\in\X$.

Let $\rho\colon\X\to[-\infty,+\infty]$ be a function. The \emph{effective domain} and \emph{epigraph} of $\rho$ are defined as
\[
\dom(\rho) \coloneqq \cb{x\in\X\mid \rho(x)<+\infty},\quad \epi(\rho)\coloneqq\{(x,z)\in\X\times \R\mid \rho(x)\leq r\},
\]
respectively. For each $r\in\R$, the corresponding \emph{lower-level set} of $\rho$ is defined as
\[
\{\rho\leq r\}\coloneqq\cb{x\in\X\mid \rho(x)\leq r}.
\]
We say that $\rho$ is \emph{proper} if $\dom(\rho)\neq \emptyset$ and $\rho(x)>-\infty$ for every $x\in\X$, \emph{convex} if $\epi(\rho)$ is convex, \emph{quasiconvex} if $\{\rho\leq r\}$ is convex for each $r\in\R$, and \emph{closed} if $\epi(\rho)$ is closed in the product topology on $\X\times\R$. Note that $\rho$ is convex if and only if $\rho(\lambda x^1+(1-\lambda)x^2)\leq \lambda \rho(x^1)+(1-\lambda) \rho(x^2)$ for every $x^1,x^2\in\dom(\rho)$ and $\lambda\in(0,1)$; $\rho$ is quasiconvex if and only if $\rho(\lambda x^1+(1-\lambda)x^2)\leq\max\{\rho(x^1),\rho(x^2)\}$ for every $x^1,x^2\in\dom(\rho)$ and $\lambda\in(0,1)$; $\rho$ is closed if and only if it is \emph{lower semicontinuous at each $x\in \X$}, that is,
\[
\rho(x)\leq \liminf_{x^\prime \rightarrow x}\rho(x^\prime):=\sup_{U\in\mathcal{N}(\X)}\inf_{x^\prime\in x+U}\rho(x^\prime)
\]
for every $x\in\X$. In the latter case, we indeed have $\rho(x)=\liminf_{x^\prime\rightarrow x}\rho(x^\prime)$ for every $x\in\X$. We also say that $\rho$ is \emph{concave} if $-\rho$ is convex, \emph{quasiconcave} if $-\rho$ is quasiconvex, and $\rho$ is \emph{upper semicontinuous at each $x\in\X$} if $-\rho$ is lower semicontinuous at each $x\in\X$.

The function $\rho^\ast\colon\X^\ast \to [-\infty,+\infty]$ defined by
\[
\rho^\ast(x^\ast)\coloneqq \sup_{x\in\X}\of{\ip{x^\ast,x}-\rho(x)},\quad x^\ast\in\X^\ast,
\]
is called the \emph{conjugate function} or \emph{Legendre-Fenchel transform} of $\rho$. Then, the \emph{biconjugate function} $\rho^{\ast\ast}\colon \X\to[-\infty,+\infty]$ of $\rho$ is defined by
\[
\rho^{\ast\ast}(x)\coloneqq \sup_{x^\ast\in\X^\ast}\of{\ip{x^\ast,x}-\rho^\ast(x^\ast)},\quad x\in\X.
\]
It is easy to see that if $\rho^{\ast\ast}$ is proper, then so is $\rho^\ast$; if $\rho^\ast$ is proper, then so is $\rho$.

We recall the celebrated Fenchel-Moreau biconjugation theorem next.

\begin{theorem}\label{thm:FM}
	\citep[Theorems~2.3.3,~2.3.4]{zalinescu} Let $\rho\colon \X\to[-\infty,+\infty]$ be a function. The following are equivalent:
	\begin{enumerate}[(i)]
		\item $\rho$ is a proper closed convex function, or $\rho\equiv +\infty$, or $\rho\equiv -\infty$.
		\item $\rho=\rho^{\ast\ast}$, that is, $\rho(x)=\sup_{x^\ast\in \X^\ast}\of{\ip{x^\ast,x}-\rho^\ast(x^\ast)}$ for each $x\in\X$.
	\end{enumerate}
\end{theorem}

\subsection{Complete lattices of sets}\label{sec:lattice}

Let $\Z$ be a real linear space. We denote by $2^\Z$ the power set of $\Z$, that is, the set of all subsets of $\Z$ including the empty set $\emptyset$ and the full space $\Z$ itself. For a set $A\subseteq\Z$, its \emph{convex hull} is denoted by $\co(A)$ and its convex-analytic \emph{indicator function} $I_{A}\colon \Z\to[0,+\infty]$ is defined by
\[
I_{A}(z)\coloneqq \begin{cases} 0 &  \text{if }z\in A,\\ +\infty & \text{if }z\in A^c\coloneqq \Z\setminus A.\end{cases}
\]
Note that $A$ is a convex set if and only if $I_A$ is a convex function. The next lemma is a less trivial characterization of convex sets in terms indicator functions, which will be crucial in the proof of our main result.

\begin{lemma}\label{lem:ind}
	Let $A,B\subseteq \Z$ be convex sets with $A\subseteq B$. The set $B\setminus A$ is convex if and only if $I_{A}$ is a quasiconcave function on $B$.
\end{lemma}

\begin{proof}
	Let $z^1,z^2\in B$ and $\lambda\in(0,1)$. Suppose that $z^1\in A$ or $z^2\in A$. Then, $I_{A}(z^1)=0$ or $I_{A}=0$. Hence, $I_{A}(\lambda z^1+(1-\lambda)z^2)\geq \min\{I_{A}(z^1), I_{A}(z^2)\}=0$ holds trivially. Therefore, $I_{A}$ is quasiconcave on $B$ if and only if $I_{A}(\lambda z^1+(1-\lambda)z^2)\geq \min\{I_{A}(z^1), I_{A}(z^2)\}$ for every $z^1\in B\setminus A$ and $z^2\in B\setminus A$. In the latter condition, we have $\min\{I_{A}(z^1), I_{A}(z^2)\}=+\infty$. Hence, $I_{A}$ is quasiconcave on $B$ if and only if $B\setminus A$ is convex.
\end{proof}

For a family $(A_j)_{j\in J}$ of subsets of $\Z$, where $J$ is an arbitrary nonempty index set, it is easy to verify that, for every $z\in\Z$, it holds
\begin{equation}\label{indicatorcalculus}
	I_{\underset{j\in J}{\bigcup}A_j}(z)=\inf_{j\in J}I_{A_j}(z),\quad I_{\underset{j\in J}{\bigcap}A_j}(z)=\sup_{j\in J}I_{A_j}(z).
\end{equation}

Let $A,B\subseteq\Z$ be given. An immediate observation yields that
\[
I_{A\cap B}(z)=I_A(z)+I_B(z),\quad z\in\Z.
\]
Moreover, the Minkowski sum of $A$ and $B$ is defined as
\[
A+B\coloneqq\cb{z^1+z^2\mid z^1\in A, z^2\in B}
\]
with the convention that $A+\emptyset\coloneqq\emptyset+B\coloneqq\emptyset$. Given $z\in\Z$, we define $z+A\coloneqq \{z\}+A$. For $\lambda\in\R$ and $A\subseteq\Z$, we define $\lambda A\coloneqq\cb{\lambda z\mid z\in A}$ with the convention that $\lambda \emptyset =\emptyset$. A nonempty set $\K\subseteq \Z$ is said to be a \emph{cone} if $\lambda \K=\K$ for every $\lambda>0$. Given a cone $\K\subseteq\Z$, the set $A\subseteq\Z$ is said to be \emph{$\K$-monotone} if $A+\K=A$. 

Let $\leq$ be a reflexive transitive relation on $\Z$. We say that $\Z$ is a \emph{preordered linear space} with respect to $\leq$ if $z^1\leq z^2$ implies $\lambda z^1+z\leq \lambda z^2+z$ for every $z^1,z^2,z\in\Z$ and $\lambda>0$. In this case, $\leq$ is determined uniquely by the convex cone
\[
\Z_+\coloneqq\cb{z\in\Z\mid 0\leq z}
\]
of \emph{positive} elements. In particular, for every $z^1,z^2\in\Z$,
\[
z^1\leq z^2\quad\Leftrightarrow \quad z^2\in z^1+\Z_+.
\]
We also define the cone $\Z_-\coloneqq -\Z_+$ of \emph{negative} elements.

We denote by $\P_+(\Z)$ the set of all $\Z_+$-monotone subsets of $\Z$, that is,
\[
\P_+(\Z)=\{A\subseteq \Z\mid A = A + \Z_+\}.
\]
The set $\P_+(\Z)$ is an order-complete lattice with respect to the partial order $\supseteq$; for a collection $\A$ of sets in $\P_+(\Z)$, the corresponding infimum and supremum are given by
\[
\inf_{\P_+(\Z)} \A=\bigcup_{A\in\A}A,\quad \sup_{\P_+(\Z)}\A=\bigcap_{A\in\A}A,
\]
respectively.

Suppose that $\Z$ is a topological linear space. We assume that the preorder $\leq$ is \emph{upper semicontinuous} in the sense that the convex cone $\Z_+$ is closed with respect to the topology on $\Z$. For a set $A\subseteq\Z$, the closure of $A$ is denoted by $\cl(A)$. Let $\F_+(\Z)$ denote the set of all $\Z_+$-monotone closed subsets of $\Z$, that is,
\[
\F_+(\Z)=\{A\subseteq \Z\mid A = \cl(A+\Z_+)\}.
\]
Similar to $\P_+(\Z)$, $\F_+(\Z)$ is an order-complete lattice with respect to $\supseteq$ with infimum and supremum formulae given by
\[
\inf_{\F_+(\Z)} \A=\cl\of{\bigcup_{A\in\A}A},\quad \sup_{\F_+(\Z)}\A=\bigcap_{A\in\A}A,
\]
respectively, for every $\A\subseteq\F_+(\Z)$.

Suppose further that $\Z$ is a Hausdorff locally convex topological linear space. In this case, we denote by $\Z^\ast$ the topological dual space of $\Z$ and by $\ip{\cdot,\cdot}\colon \Z^\ast\times \Z\to \R$ the bilinear duality mapping between $\Z^\ast$ and $\Z$. For $z^\ast\in\Z^\ast$ and $r\in\R$, we define the halfspace
\[
H(z^\ast,r)\coloneqq\cb{z\in\X\mid\ip{z^\ast,z}\geq r}.
\]
If $r=0$, then we say that the halfspace is \emph{homogeneous}. Given a cone $\K\subseteq\Z$, the \emph{positive dual cone} of $\K$ is defined as
\[
\K^+\coloneqq\cb{z^\ast\in \Z^\ast \mid\forall z\in\K\colon \ip{z^\ast, z}\geq 0},
\]
which is a closed convex cone in $\Z^\ast$ under the weak$^\ast$ topology $\sigma(\Z^\ast,\Z)$. We write $\Z_+^+\coloneqq (\Z_+)^+$ if $\K=\Z_+$.

For a set $A\subseteq \Z$, we define its support function $\sigma_A\colon \Z^\ast\to [-\infty,+\infty]$ by
\[
\sigma_A (z^\ast) \coloneqq \inf_{z\in A}\ip{z^\ast,z},
\]
with the convention that $\sigma_\emptyset(z^\ast)=+\infty$ for each $z^\ast\in\Z^\ast$. Let $\K$ be a cone. If $A$ is a nonempty $\K$-monotone set, then it can be checked that $\sigma_A(z^\ast)=-\infty$ for every $z^\ast\in\X^\ast\setminus\K^+$. Moreover, as a result of the well-known separation theorem for convex sets, $A$ is a $\K$-monotone closed convex set if and only if
\begin{align}\label{separation}
	A=\bigcap_{z^\ast\in\K^+\setminus\{0\}}H(z^\ast,\sigma_{A}(z^\ast))=\bigcap_{z^\ast\in\K^+\setminus\{0\}}\cb{z\in\Z\mid\ip{z^\ast,z}\geq \sigma_{A}(z^\ast)}.
\end{align}
Next, let us consider the special case $\K=\Z_+$. The set of all $\Z_+$-monotone closed convex subsets of $\Z$ is denoted by $\G_+(\Z)$, that is,
\[
\G_+(\Z)=\{A\subseteq\Z\mid A=\cl\co(A+\Z_+)\}.
\]
Similar to $\P_+(\Z)$ and $\F_+(\Z)$, $\G_+(\Z)$ is an order-complete lattice with respect to $\supseteq$ with infimum and supremum formulae given by
\[
\inf_{\G_+(\Z)} \A=\cl\co\of{\bigcup_{A\in\A}A},\quad \sup_{\G_+(\Z)}\A=\bigcap_{A\in\A}A,
\]
respectively, for every $\A\subseteq\G_+(\Z)$.

\subsection{Set-valued functions}\label{sec:function}

Let $\X, \Z$ be preordered real linear spaces whose preorders are upper semicontinuous. With a slight abuse of notation, we denote by $\leq$ both of these preorders. Let $R \colon \X\to 2^\Z$ be a set-valued function. We define the \emph{effective domain} and \emph{graph} of $R$ as
\[
\dom(R)\coloneqq\{x\in\X\mid R(x)\neq\emptyset\},\quad \gr(R)\coloneqq \cb{(x,z)\in\X\times\Z\mid z\in R(x)},
\]
respectively. We say that $R$ is \emph{proper} if $\dom(R)\neq\emptyset$ and $R(x)\neq\Z$ for every $x\in\X$. We define the inverse $R^{-1}\colon\Z\to 2^\X$ of $R$ by
\[
R^{-1}(z)\coloneqq \cb{x\in\X\mid z\in R(x)},\quad z\in\Z.
\]
It is immediate that $(R^{-1})^{-1}=R$, that is,
\[
R(x)=\cb{z\in \Z\mid x\in R^{-1}(z)},\quad x\in\X.
\]

$R$ is said to be \emph{increasing} (resp. \emph{decreasing}) if $x^1\leq x^2$ implies $R(x^1)\supseteq R(x^2)$ (resp. $R(x^1)\subseteq R(x^2)$) for every $x^1,x^2\in\X$. By symmetry, these monotonicity properties can also be defined for $R^{-1}$. The following result formulates the relationship between the monotonicity of $F$ and that of the values of $R^{-1}$ with respect to $\X_+$.

\begin{lemma}\label{lem:dec}
	\citep[Proposition~4]{completeduality} Let $R\colon\X\to 2^{\Z}$ be a set-valued function. Then, $R(x)\in\P_+(\Z)$ for every $x\in \X$ if and only if $R^{-1}$ is decreasing. Moreover, $R$ is decreasing if and only if $R^{-1}(z)\in\P_+(\X)$ for every $z\in\Z$.
\end{lemma}

Let us assume that $R$ maps into $\P_+(\Z)$. We say that $R$ is \emph{convex} if
\[
R(\lambda x^1+(1-\lambda) x^2)\supseteq \lambda R(x^1)+(1-\lambda) R(x^2)
\]
for every $x^1,x^2\in\X$ and $\lambda\in [0,1]$. If $R$ is decreasing, then we can also define the convexity of $R^{-1}$ in a similar way.

\begin{lemma}\label{lem:convexity}
	\citep[Propositions~3,~4]{completeduality} Let $R\colon\X\to\P_+(\Z)$ be a set-valued function. Then, $R$ is convex if and only if $\gr R$ is convex. If $R$ is further assumed to be decreasing, then these properties are also equivalent to the convexity of $R^{-1}$.
\end{lemma}

From now on, we assume that $\Z$ is a Hausdorff locally convex topological linear space. For each $z^\ast\in\Z^\ast$, the function $\varphi_{R,z^\ast}\colon\X\to[-\infty,+\infty]$ defined by
\begin{equation}\label{eq:sc}
	\varphi_{R,z^\ast}(x)\coloneqq \sigma_{R(x)}(z^\ast)=\inf_{z\in R(x)}\ip{z^\ast,z},\quad x\in\X,
\end{equation}
is called a \emph{(linear) scalarization} of $R$.

\begin{remark}\label{rem:scalarization}
	For each $x\in\X$ and $z^\ast\in\Z^\ast$, since $R(x)\in\P_+(\Z)$, it is easy to check that $\varphi_{R,z^\ast}(x)=-\infty$ if $z^\ast\notin \Z_+^+$.
\end{remark}

The next result characterizes the properness/convexity of a set-valued function in terms of the properness/convexity of its scalarizations as extended real-valued functions.

\begin{lemma}\label{lem:scalarization}
	\citep[Lemma 4.20]{setoptsurv} Let $R\colon\X\to\G_+(\Z)$ be a set-valued function. Then, the following results hold:
	\begin{enumerate}[(i)]
		\item $R$ is proper if and only if there exists $z^\ast\in \Z_+^+\setminus\{0\}$ such that $\varphi_{R,z^\ast}$ is proper.
		\item $R$ is convex if and only if $\varphi_{R,z^\ast}$ is a convex function for each $z^\ast\in\Z^\ast$.
	\end{enumerate}
\end{lemma}

For future use, let us introduce the set
\begin{equation}\label{eq:ZR}
	\Z^\ast_R\coloneqq\cb{z^\ast\in\Z^\ast\setminus\{0\}\mid \varphi_{R,z^\ast}\text{ is proper}}.
\end{equation}
Note that $\Z^\ast_R$ is a cone and $\Z^\ast_R\subseteq \Z^+_+\setminus\{0\}$.

The set-valued function $R\colon \X\to\P_+(\Z)$ is called \emph{closed-valued} if $R(x)\in\F_+(\Z)$ for every $x\in\X$, \emph{lower level-closed} if $R^{-1}(z)$ is a closed set for every $z\in\Z$, \emph{closed} if $\gr R$ is a closed set with respect to the product topology on $\X\times\Z$. Clearly, if $R$ is closed, then it is closed-valued and lower level-closed. Let $\mathcal{N}(\X)$ be a neighborhood base of $0\in\X$. Suppose that $R$ is closed-valued. Given $x\in\X$, $R$ is called \emph{lattice-lower semicontinuous at $x$} if
\begin{align*}
R(x)\supseteq \liminf_{x^\prime\rightarrow x}R(x^\prime)\coloneqq &\sup_{\F_+(\Z)}\cb{\inf_{\F_+(\Z)} \cb{R(x^\prime)\mid x^\prime\in x + U} \mid U\in\mathcal{N}(\X)}\\
=& \bigcap_{U\in\mathcal{N}(\X)}\cl\of{\bigcup_{x^\prime\in x+U}R(x^\prime)}.
\end{align*}
In this case, we indeed have $R(x)=\liminf_{x^\prime\rightarrow x}R(x^\prime)$. Then, $R$ is called \emph{lattice-lower semicontinuous} if it is lattice-lower semicontinuous at every $x\in\X$. The function $R$ is called \emph{scalarly lower semicontinuous} (at $x$) if $\varphi_{R,z^\ast}$ is lower semicontinuous (at $x$) for each $z^\ast\in\Z^+_+\setminus\{0\}$.

\begin{lemma}\label{lem:closedness}
	\citep[Proposition~4.9]{setoptsurv} Let $R\colon \X\to\F_+(\Z)$ be a set-valued function. Then, it is closed if and only if it is lattice-lower semicontinuous.
\end{lemma}

Unlike the case of convexity (see Lemma~\ref{lem:scalarization}), lattice-lower semicontinuity of a set-valued function is not equivalent to the lower semicontinuity of its scalarizations. For functions with closed convex values, only a one-way implication holds as stated by the next result.

\begin{lemma}\label{lem:scalarization2}
	\citep[Proposition~4.23]{setoptsurv} Let $R\colon \X\to\G_+(\Z)$ be a set-valued function and fix $x\in\X$. If $R$ is scalarly lower semicontinuous at $x$, then $R$ is lattice-lower semicontinuous at $x$.
\end{lemma}

Finally, we review the set-valued Fenchel-Moreau theorem. To that end, we define the \emph{(negative) conjugate function} $-R^\ast\colon\X^\ast\times\Z_+^+\setminus\{0\}\to\P_+(\Z)$ of $R$ by
\[
-R^\ast(x^\ast,z^\ast)\coloneqq\inf_{\F_+(\Z)}\cb{\cl\of{R(x)+H(z^\ast,\ip{x^\ast,-x})}\mid x\in\X}=\cl\of{\bigcup_{x\in\X}\of{R(x)+H(z^\ast,\ip{x^\ast,-x})}}
\] 
for each $x^\ast\in\X^\ast$ and $z^\ast\in\Z_+^+\setminus\{0\}$. Then, the \emph{biconjugate function} $R^{\ast\ast}\colon \X\to \P_+(\Z)$ of $R$ by
\begin{align*}
	R^{\ast\ast}(x)\coloneqq&\sup_{\F_+(\X)}\cb{-R^{\ast}(x^\ast,z^\ast)+H(z^\ast,\ip{x^\ast,x}) \mid x^\ast\in\X^\ast,\ z^\ast\in \Z_+^+\setminus\{0\}}\\
	=&\bigcap_{\substack{x^\ast\in\X^\ast,\\ z^\ast\in\Z_+^+\setminus\{0\}}}\of{-R^{\ast}(x^\ast,z^\ast)+H(z^\ast,\ip{x^\ast,x})}
\end{align*}
for each $x\in\X$.

\begin{remark}\label{rem:SVconjugate}
	\begin{enumerate}[(i)]
		\item The minus sign in $-R^\ast$ is part of the notation. Indeed, the definition of $-R^\ast$ mimics that of the negative of the conjugate function for an extended real-valued function $f\colon\X\to[-\infty,+\infty]$:
		\[
		-(f^\ast(x^\ast))=-\sup_{x\in\X}\of{\ip{x^\ast,x}-f(x)}=\inf_{x\in\X}\of{f(x)+\ip{x^\ast,-x}},\quad x^\ast\in\X^\ast.
		\]
		Similarly, the definition of $R^{\ast\ast}$ mimics that of $f^{\ast\ast}$:
		\[
		f^{\ast\ast}(x)=\sup_{x^\ast\in\X^\ast}\of{-f^\ast(x^\ast)+\ip{x^\ast,x}},\quad x\in\X.
		\]
		The main difference between the scalar and set-valued cases is that we have an extra dual variable $z^\ast\in \Z_+^+\setminus\{0\}$ in the latter case to scalarize the set-valued function. For the same reason, the bilinear form $(x^\ast,x)\mapsto \ip{x^\ast,x}$ of the scalar case is replaced with the halfspace-valued function $(z^\ast,x^\ast,x)\mapsto H(z^\ast,\ip{x^\ast,x})$. In particular, both $-R^\ast$ and $R^{\ast\ast}$ map into $\G_+(\Z)$.
		\item The conjugate and biconjugate of $R$ can be expressed in terms of those of its scalarizations. Using the definitions, it is easy to check that
		\[
		-R^\ast(x^\ast,z^\ast)=\cb{z\in\Z\mid \ip{z^\ast,z}\geq - \varphi_{R,z^\ast}^\ast(x^\ast)}
		\]
		for each $x^\ast\in\X^\ast$, $z^\ast\in\Z_+^+\setminus\{0\}$, and
		\[
		R^{\ast\ast}(x)=\bigcap_{z^\ast\in\mathcal{K}}\cb{z\in\Z\mid \ip{z^\ast,z}\geq \varphi_{R,z^\ast}^{\ast\ast}(x)}
		\]
		for each $x\in\X$, where $\mathcal{K}\subseteq\Z^\ast$ is a set such that $\{z^\ast\in\Z^\ast\setminus\{0\}\mid \varphi_{R,z^\ast}^{\ast\ast}\text{ is proper}\}\subseteq\mathcal{K}$.
	\end{enumerate}
\end{remark}

\begin{theorem}\label{thm:SVFM}
	(\citet[Theorem~2]{ham09}, \citet[Theorem~5.8]{setoptsurv}) Let $R\colon \X\to\P_+(\Z)$ be a set-valued function. The following are equivalent:
	\begin{enumerate}[(a)]
		\item $R$ is a proper closed convex set-valued function, or $R\equiv \emptyset$, or $R\equiv \Z$.
		\item $R=R^{\ast\ast}$, that is, for each $x\in\X$, we have
		\[
		R(x)=\bigcap_{\substack{x^\ast\in \X^\ast,\\  z^\ast\in \mathcal{K}}} \of{-R^\ast(x^\ast,z^\ast)+H(z^\ast,\ip{x^\ast,x})},
		\]
		where $\mathcal{K}\subseteq\Z^\ast$ is a set such that $\{z^\ast\in\Z^\ast\setminus\{0\}\mid \varphi_{R,z^\ast}^{\ast\ast}\text{ is proper}\}\subseteq\mathcal{K}$.
	\end{enumerate}
\end{theorem}

\section{Set-valued compositions}\label{sec:comp}

In this section, we consider the composition of two set-valued functions. To that end, let us fix three Hausdorff locally convex topological real linear spaces $\X,\Y,\Z$ with duals $\X^\ast,\Y^\ast,\Z^\ast$, respectively. We assume that $\Y, \Z$ are preordered linear spaces with upper semicontinuous preorders characterized by cones $\Y_+, \Z_+$ and their positive dual cones $\Y_+^+,\Z_+^+$, respectively.

Let $F\colon\Y\to\G_+(\Z)$ and $G\colon\X\to\P_+(\Y)$ be two set-valued functions. We define their composition $F\circ G\colon\X\to\G_+(\Z)$ by
\[
F\circ G(x)\coloneqq \inf_{\G_+(\Z)}\cb{F(y)\mid y\in G(x)}=\cl\co\of{\bigcup_{y\in G(x)}F(y)},\quad x\in\X.
\]

\begin{remark}\label{rem:comp}
	One can also consider the simpler forms $x\mapsto \bigcup_{y\in G(x)}F(y)\in\P_+(\Z)$ and $x\mapsto \cl(\bigcup_{y\in G(x)}F(y)) \in \F_+(\Z)$. We note that all three alternatives have the same conjugate and biconjugate functions. Since our focus will be on conjugation and duality, we prefer working with the current definition for which the composition takes values in $\G_+(\Z)$; see Remark~\ref{rem:SVconjugate}(i).
\end{remark}

\begin{proposition}\label{prop:propcomp}
	\begin{enumerate}[(i)]
		\item Suppose that $\X$ is also a preordered linear space. If $G$ is decreasing, then $F\circ G$ is decreasing.
		\item If $F$ and $G$ are convex, then $F\circ G$ is convex.
	\end{enumerate}
\end{proposition}

\begin{proof}
	\begin{enumerate}[(i)]
		\item Let $x^1,x^2\in\X$ with $x^1\leq x^2$. Let $z\in F(y)$ for some $y\in G(x^1)$. Since $G$ is decreasing, we have $y\in G(x^1)\subseteq G(x^2)$. Then, $z\in F\circ G(x^2)$. It follows that $F\circ G(x^1)\subseteq F\circ G(x^2)$.
		\item Let $x^1,x^2\in\X$ and $\lambda\in(0,1)$. Let $z^1\in F(y^1)$ and $z^2\in F(y^2)$ for some $y^1\in G(x^1)$ and $y^2\in G(x^2)$. Since $G$ is convex, we have
		\[
		\lambda y^1+(1-\lambda)y^2\in \lambda G(x^1)+(1-\lambda)G(x^2)\subseteq G(\lambda x^1+(1-\lambda)x^2).
		\]
		Moreover, since $F$ is convex, we have
		\[
		\lambda z^1+(1-\lambda)z^2\in \lambda F(y^1)+(1-\lambda)F(y^2)\subseteq F(\lambda y^1+(1-\lambda)y^2).
		\]
		Hence,
		\[
		\lambda z^1+(1-\lambda)z^2\in \bigcup_{y\in G(\lambda x^1+(1-\lambda)x^2)}F(y).
		\]
		Then, by elementary properties of Minkowski sums, we obtain
		\begin{align*}
			&\lambda F\circ G(x^1)+(1-\lambda) F\circ G(x^2)\\
			&=
			\lambda \cl\co\of{\bigcup_{y^1\in G(x^1)}F(y^1)}+(1-\lambda) \cl\co\of{\bigcup_{y^2\in G(x^2)}F(y^2)}\\
			&=\cl\of{\lambda \co\of{\bigcup_{y^1\in G(x^1)}F(y^1)}+(1-\lambda) \co\of{\bigcup_{y^2\in G(x^2)}F(y^2)}}\\
			&=\cl\co\of{\lambda \bigcup_{y^1\in G(x^1)}F(y^1)+(1-\lambda)\bigcup_{y^2\in G(x^2)}F(y^2)}\\
			&\subseteq \cl\co\of{\bigcup_{y\in G(\lambda x^1+(1-\lambda)x^2)}F(y)}\\
			&= F\circ G(\lambda x^1+(1-\lambda)x^2),
		\end{align*}
		which completes the proof.
	\end{enumerate}
\end{proof}

\begin{proposition}\label{prop:comp-sc}
	Let $z^\ast\in\Z^+_+\setminus\{0\}$ and $x\in\X$. Then, we have
	\[
	\varphi_{F\circ G,z^\ast}(x)=\inf_{y\in G(x)}\varphi_{F,z^\ast}(y).
	\]
\end{proposition}

\begin{proof}
	Since the support function of a set is the same as that of its closed convex hull, we have
	\[
	\varphi_{F\circ G,z^\ast}(x)=\inf_{z\in F\circ G(x)}\ip{z^\ast,z}=\inf_{z\in \bigcup_{y\in G(x)}F(y)}\ip{z^\ast,z}=\inf_{y\in G(x)}\inf_{z\in F(y)}\ip{z^\ast,z}=\inf_{y\in G(x)}\varphi_{F,z^\ast}(y).
	\]
	Hence, the result follows.
\end{proof}

Based on Proposition~\ref{prop:comp-sc}, we can make some simple observations about the properness of scalarizations, as the next corollary states.

\begin{corollary}\label{cor:comp-sc}
	Let $z^\ast\in\Z^+_+\setminus\{0\}$ be such that $\varphi_{F\circ G,z^\ast}$ is proper. Then, $\varphi_{F,z^\ast}$ is proper.
\end{corollary}

\begin{proof}
	By Lemma~\ref{prop:comp-sc}, we must have $\varphi_{F,z^\ast}(y)>-\infty$ for every $(x,y)\in\gr(G)$ and there exists $(x^0,y^0)\in \gr G$ such that $\varphi_{F,z^\ast}(y^0)<+\infty$. This implies that $\varphi_{F,z^\ast}$ is proper.
\end{proof}

Our aim is to provide a formula for the set-valued conjugate of $F\circ G$. In view of Remark~\ref{rem:SVconjugate}(ii), it is sufficient to calculate the conjugates of the scalarizations of $F\circ G$; see \eqref{eq:sc}. Our calculation will follow a minimax argument that makes use of a compactness assumption and an unboundedness/monotonicity assumption, which we introduce next.

\begin{assumption}\label{asmp:conegen}
	The cone $\Y^+_+$ has a convex and weak*-compact cone generator, that is, there exists a convex and $\sigma(\Y^\ast,\Y)$-compact set $\mathcal{B}_{\Y^\ast}$ such that every $y^\ast\in\Y^+_+\setminus\{0\}$ can be written as $y^\ast=\lambda \bar{y}^\ast$ for some $\lambda>0$ and $\bar{y}^\ast\in \mathcal{B}_{Y^\ast}$.
\end{assumption}

\begin{assumption}\label{asmp:sc}
	One of the following conditions holds:
	\begin{enumerate}[(a)]
		\item For each $y^\ast\in\mathcal{B}_{\Y^\ast}\setminus\{0\}$, we have $\inf_{x\in\X}\varphi_{G,y^\ast}(x)=-\infty$.
		\item $\X$ is a preordered linear space with upper semicontinuous preoder with cone $\X_+$ and its positive dual cone $\X_+^+$. The cone $\X^{\sharp}_+\coloneqq \{x\in\X\mid\forall x^\ast\in\X_+^+\setminus\{0\}\colon \ip{x^\ast,x}>0\}$ is nonempty. For each $y^\ast\in\mathcal{B}_{\Y^\ast}\setminus\{0\}$, the function $\varphi_{G,y^\ast}$ is strictly decreasing, i.e., for every $x^1,x^2\in\X$, we have
		\[
		x^2\in x^1+\X_+^{\sharp}\quad\Leftrightarrow\quad \varphi_{G,y^\ast}(x^1)>\varphi_{G,y^\ast}(x^2).
		\]
	\end{enumerate}
\end{assumption}

\begin{remark}\label{rem:asmp}
	Assumptions~\ref{asmp:conegen},~\ref{asmp:sc}(b) have also appeared in \citet[Section~4]{mucahitthesis} in the context of scalar quasiconvex compositions.
\end{remark}

We proceed with the main theorem of the paper.

\begin{theorem}\label{thm:conj}
	Suppose that Assumptions~\ref{asmp:conegen},~\ref{asmp:sc} hold. Let $F\colon \Y\to\G_+(\Z)$, $G\colon \X\to\G_+(\Y)$ be convex and scalarly lower semicontinuous set-valued functions. Then, for each $z^\ast\in \Z^\ast_{F\circ G}$, we have
	\[
	\varphi_{F\circ G,z^\ast}^\ast(x^\ast)=\inf_{y^\ast\in\Y^\ast_G}\of{\varphi^\ast_{G,y^\ast}(x^\ast)+\varphi^\ast_{F,z^\ast}(y^\ast)};
	\]
	in particular,
	\[
	-(F\circ G)^\ast(x^\ast,z^\ast)=\bigcap_{y^\ast\in\Y^\ast_G}\cb{z\in\Z\mid \ip{z^\ast,z}\geq -\varphi^\ast_{G,y^\ast}(x^\ast)-\varphi^\ast_{F,z^\ast}(y^\ast)}
	\]
	for each $x^\ast\in\X^\ast$. (Here, $\Y^\ast_G\coloneqq\{y^\ast\in\Y^\ast\setminus\{0\}\mid\varphi_{G,y^\ast}\text{ is proper}\}$.)
\end{theorem}

The proof of Theorem~\ref{thm:conj} will be given in Section~\ref{sec:proof}.

When the composition is guaranteed to be a proper scalarly lower semicontinuous set-valued function, we obtain a dual representation for it as a corollary of Theorem~\ref{thm:conj}.

\begin{corollary}\label{cor:biconj}
	Suppose that Assumptions~\ref{asmp:conegen},~\ref{asmp:sc} hold. Let $F\colon \Y\to\G_+(\Z)$, $G\colon \X\to\G_+(\Y)$ be convex and scalarly lower semicontinuous set-valued functions. Then,
	\[
	(F\circ G)^{\ast\ast}(x)=\bigcap_{\substack{x^\ast\in\X^\ast,\\
			y^\ast\in\Y^\ast_G,\\
			z^\ast\in\Z^\ast_F}}\cb{z\in\Z\mid \ip{z^\ast,z}\geq \ip{x^\ast,x}-\varphi^\ast_{G,y^\ast}(x^\ast)-\varphi^\ast_{F,z^\ast}(y^\ast)}
	\]
	for each $x\in\X$. Moreover, if $F\circ G$ is a proper scalarly closed set-valued function, then
	\[
	F\circ G(x)=\bigcap_{\substack{x^\ast\in\X^\ast,\\
			y^\ast\in\Y^\ast_G,\\
			z^\ast\in\Z^\ast_F}}\cb{z\in\Z\mid \ip{z^\ast,z}\geq \ip{x^\ast,x}-\varphi^\ast_{G,y^\ast}(x^\ast)-\varphi^\ast_{F,z^\ast}(y^\ast)}
	\]
	for each $x\in\X$.
\end{corollary}

\begin{proof}
	Let $z^\ast\in\Z^+_+\setminus\{0\}$. If $\varphi^{\ast\ast}_{F\circ G,z^\ast}$ is proper, then so are $\varphi^{\ast}_{F\circ G,z^\ast}$, $\varphi_{F\circ G,z^\ast}$, and $\varphi_{F,z^\ast}$ by Corollary~\ref{cor:comp-sc}. Hence, we may apply Remark~\ref{rem:SVconjugate}(ii) with $\mathcal{K}=\Z^\ast_F$; see \eqref{eq:ZR}. Then, the first formula follows as a direct consequence of Theorem~\ref{thm:conj}. Note that $F\circ G$ is a convex set-valued function by Proposition~\ref{prop:propcomp}(ii). If $F\circ G$ is scalarly closed, then it is closed by Lemmata~\ref{lem:closedness},~\ref{lem:scalarization2}. Hence, if $F\circ G$ is proper and scalarly closed, then we have $R=R^{\ast\ast}$ by Theorem~\ref{thm:SVFM}. Therefore, the second formula follows from the first formula.
\end{proof}

\section{Proof of Theorem~\ref{thm:conj}}\label{sec:proof}

The aim of this section is prove Theorem~\ref{thm:conj}. The proof will rely on Liu's minimax inequality; see \citet{liu}, \citet[Corollary~11]{liurelated1}, \citet[Theorem~3.1]{liurelated2} instead of Sion's standard minimax inequality \citep[Corollary~3.3]{sion}. 

We start by a technical preparation for the proof. Recall that we work under Assumptions~\ref{asmp:conegen},~\ref{asmp:sc} and we fix convex scalarly lower semicontinuous set-valued functions $F\colon \Y\to\G_+(\Z)$, $G\colon \X\to\P_+(\Y)$.

Let us fix $x^\ast\in\X^\ast$, $z^\ast\in\Z_+^+\setminus\{0\}$ and define two functions $f,\tilde{f}\colon \X\times\Y\times\mathcal{B}_{\Y^\ast}\to[-\infty,+\infty]$ by
\begin{align*}
	& f(x,y,y^\ast)\coloneqq \ip{x^\ast,x}-\varphi_{F,z^\ast}(y)- I_{A(y^\ast)}(x,y),\\
	& \tilde{f}(x,y,y^\ast)\coloneqq \ip{x^\ast,x}-\varphi_{F,z^\ast}(y)- I_{\tilde{A}(y^\ast)}(x,y)
\end{align*}
for each $(x,y,y^\ast)\in\X\times\Y\times\mathcal{B}_{\Y^\ast}$, where
\begin{equation}\label{eq:A}
	A(y^\ast)\coloneqq \cb{(x,y)\in\X\times\Y\mid \ip{y^\ast,y}\geq \varphi_{G,y^\ast}(x)},\quad
	\tilde{A}(y^\ast)\coloneqq \cb{(x,y)\in\X\times\Y\mid \ip{y^\ast,y}> \varphi_{G,y^\ast}(x)}
\end{equation}
for each $y^\ast\in \mathcal{B}_{\Y^\ast}$. Since $\tilde{A}(y^\ast)\subseteq A(y^\ast)$, we have $f(x,y,y^\ast)\geq \tilde{f}(x,y,y^\ast)$ for each $(x,y,y^\ast)\in\X\times\Y\times\mathcal{B}_{\Y^\ast}$.

\begin{lemma}\label{lem:proof1}
	The following results hold:
	\begin{enumerate}[(i)]
		\item For each $y^\ast\in\mathcal{B}_{\Y^\ast}$, the function $(x,y)\mapsto f(x,y,y^\ast)$ is concave and upper semicontinuous function on $\X\times\Y$.
		\item For each $(x,y)\in\X\times\Y$, the function $y^\ast\mapsto f(x,y,y^\ast)$ is quasiconvex on $\mathcal{B}_{\Y^\ast}$.
		\item For each $y^\ast\in\mathcal{B}_{\Y^\ast}$, the function $(x,y)\mapsto \tilde{f}(x,y,y^\ast)$ is concave on $\X\times\Y$.
		\item For each $(x,y)\in\X\times\Y$, the function $y^\ast\mapsto \tilde{f}(x,y,y^\ast)$ is quasiconvex and weak*-lower semicontinuous on $\mathcal{B}_{\Y^\ast}$.
	\end{enumerate}
\end{lemma}

\begin{proof}
	\begin{enumerate}[(i)]
		\item Let $y^\ast\in \mathcal{B}_{\Y^\ast}$. Since $G$ is scalarly lower semicontinuous and convex, $\varphi_{G,y^\ast}$ is a convex lower semicontinuous function on $\X$. Moreover, $y\mapsto\ip{y^\ast,y}$ is a continuous linear function on $\Y$. It follows that $A(y^\ast)$ is a closed convex subset of $\X\times\Y$. Hence, $I_{A(y^\ast)}$ is a convex lower semicontinuous function on $\X\times\Y$. Since $F$ is scalarly lower semicontinuous and convex, $\varphi_{F,z^\ast}$ is a convex lower semicontinuous function on $\Y$. Moreover, $x\mapsto \ip{x^\ast,x}$ is a continuous linear function on $\X$. Therefore, $(x,y)\mapsto f(x,y,y^\ast)$ is a concave upper semicontinuous function on $\X\times\Y$.
		
		\item Let $(x,y)\in\X\times\Y$. We have
		\[
		I_{A(y^\ast)}(x,y)=I_{B(x,y)}(y^\ast),
		\]
		where
		\[
		B(x,y)\coloneqq\cb{y^\ast\in \mathcal{B}_{\Y^\ast}\mid \ip{y^\ast,y}\geq \sigma_{G(x)}(y^\ast)}.
		\]
		We show that $\mathcal{B}_{\Y^\ast}\setminus B(x,y)$ is a convex set. Let $y^{\ast,1},y^{\ast,2}\in \mathcal{B}_{\Y^\ast}\setminus B(x,y)$ and $\lambda \in (0,1)$. Hence, $\ip{y^\ast,y}<\sigma_{G(x)}(y^{\ast,1})$ and $\ip{y^\ast,y}<\sigma_{G(x)}(y^{\ast,1})$. Note that $\sigma_{G(x)}$ is a concave function as a supremum of linear functions. Hence,
		\begin{align*}
			\ip{\lambda y^{\ast,1}+(1-\lambda)y^{\ast,2},y}&=\lambda\ip{y^{\ast,1},y}+(1-\lambda)\ip{y^{\ast,2},y}\\
			&<\lambda \sigma_{G(x)}(y^{\ast,1})+(1-\lambda)\sigma_{G(x)}(y^{\ast,2})\\
			&\leq \sigma_{G(x)}(\lambda y^{\ast,1}+(1-\lambda)y^{\ast,2})
		\end{align*}
		so that $\lambda y^{\ast,1}+(1-\lambda)y^{\ast,2}\in \mathcal{B}_{y^\ast}\setminus B(x,y)$.
		Therefore, $I_{B(x,y)}$ is a quasiconcave function on $\mathcal{B}_{\Y^\ast}$ by Lemma~\ref{lem:ind}. This implies that $y^\ast\mapsto f(x,y,y^\ast)$ is a quasiconvex function on $\mathcal{B}_{\Y^\ast}$.
		
		\item Let $y^\ast\in \mathcal{B}_{\Y^\ast}$. Since $G$ is convex, $\varphi_{G,y^\ast}$ is a convex function on $\X$. Moreover, $y\mapsto\ip{y^\ast,y}$ is a linear function on $\Y$. It follows that $\tilde{A}(y^\ast)$ is a convex subset of $\X\times\Y$. Hence, $I_{\tilde{A}(y^\ast)}$ is a convex function on $\X\times\Y$. Since $F$ is convex, $\varphi_{F,z^\ast}$ is a convex function on $\Y$. Moreover, $x\mapsto \ip{x^\ast,x}$ is a linear function on $\X$. Therefore, $(x,y)\mapsto \tilde{f}(x,y,y^\ast)$ is a concave function on $\X\times\Y$.
		
		\item Let $(x,y)\in\X\times\Y$. We have
		\[
		I_{\tilde{A}(y^\ast)}(x,y)=I_{\tilde{B}(x,y)}(y^\ast),
		\]
		where
		\[
		\tilde{B}(x,y)\coloneqq\cb{y^\ast\in \mathcal{B}_{\Y^\ast}\mid \ip{y^\ast,y}> \sigma_{G(x)}(y^\ast)}.
		\]
		We show that $\mathcal{B}_{\Y^\ast}\setminus \tilde{B}(x,y)$ is a convex set. Let $y^{\ast,1},y^{\ast,2}\in \mathcal{B}_{\Y^\ast}\setminus \tilde{B}(x,y)$ and $\lambda \in (0,1)$. Hence, $\ip{y^\ast,y}\leq\sigma_{G(x)}(y^{\ast,1})$ and $\ip{y^\ast,y}\leq\sigma_{G(x)}(y^{\ast,1})$. Note that $\sigma_{G(x)}$ is a concave function as a supremum of linear functions. Hence,
		\begin{align*}
			\ip{\lambda y^{\ast,1}+(1-\lambda)y^{\ast,2},y}&=\lambda\ip{y^{\ast,1},y}+(1-\lambda)\ip{y^{\ast,2},y}\\
			&\leq\lambda \sigma_{G(x)}(y^{\ast,1})+(1-\lambda)\sigma_{G(x)}(y^{\ast,2})\\
			&\leq \sigma_{G(x)}(\lambda y^{\ast,1}+(1-\lambda)y^{\ast,2})
		\end{align*}
		so that $\lambda y^{\ast,1}+(1-\lambda)y^{\ast,2}\in \mathcal{B}_{y^\ast}\setminus \tilde{B}(x,y)$.
		Therefore, $I_{\tilde{B}(x,y)}$ is a quasiconcave function on $\mathcal{B}_{\Y^\ast}$ by Lemma~\ref{lem:ind}. Moreover, $\sigma_{G(x)}$ is a weak*-upper semicontinuous function as an infimum of weak*-continuous functions. Hence, $\tilde{B}(x,y)$ is a weak*-open set so that $I_{\tilde{B}(x,y)}$ is a weak*-upper semicontinuous function. It follows that $y^\ast\mapsto\tilde{f}(x,y,y^\ast)$ is a quasiconvex weak*-lower semicontinuous function on $\mathcal{B}_{\Y^\ast}$.
	\end{enumerate}
\end{proof}

\begin{lemma}\label{lem:proof2}
	Let $y^\ast\in\mathcal{B}_{\Y^\ast}$. Then, we have the following results:
	\begin{enumerate}[(i)]
		\item Let $y\in\Y$ and define
		\[
		A_y(y^\ast)\coloneqq\{x\in\X\mid (x,y)\in A(y^\ast)\},\quad \tilde{A}_y(y^\ast)\coloneqq\{x\in\X\mid (x,y)\in \tilde{A}(y^\ast)\}.
		\]
		Then, it holds $A_y(y^\ast)=\cl(\tilde{A}_y(y^\ast))$.
		\item It holds \[
		\sup_{\substack{x\in\X,\\ y\in\Y}}f(x,y,y^\ast)=\sup_{\substack{x\in\X,\\ y\in\Y}}\tilde{f}(x,y,y^\ast).
		\]
	\end{enumerate}
\end{lemma}

\begin{proof}
	\begin{enumerate}[(i)]
		
		\item Let $y\in\Y$. Since $A(y^\ast)$ is a closed set, $A_y(y^\ast)$ is also closed as the section of a closed set. Since $\tilde{A}_y(y^\ast)\subseteq A_y(y^\ast)$ and $A_y(y^\ast)$ is a closed set, we have $\cl(\tilde{A}_y(y^\ast))\subseteq A_y(y^\ast)$.
		
		Conversely, let $x\in A_y(y^\ast)$. Hence, $\ip{y^\ast,y}\geq \varphi_{G,y^\ast}(x)$. First, suppose that Assumption~\ref{asmp:sc}(a) holds. Then, we can find $\bar{x}\in\X$ such that $\varphi_{G,y^\ast}(\bar{x})<\ip{y^\ast,y}$. Let
		\[
		x^n\coloneqq \of{1-\frac1n}x+\frac1n \bar{x},\quad n\in\mathbb{N}.
		\]
		Then, since $\varphi_{G,y^\ast}$ is a convex function, we obtain
		\[
		\varphi_{G,y^\ast}(x^n)\leq \of{1-\frac1n}\varphi_{G,y^\ast}(x)+\frac1n\varphi_{G,y^\ast}(\bar{x})<\ip{y^\ast,y},
		\]
		that is, $x^n\in \tilde{A}_y(y^\ast)$ for each $n\in\mathbb{N}$. Moreover, $(x^n)_{n\in\mathbb{N}}$ converges to $x$. Hence, $x\in\cl(\tilde{A}_y(y^\ast))$.
		
		Second, suppose that Assumption~\ref{asmp:sc}(b) holds. Let $\bar{x}\in\X_+^\sharp$ and define
		\[
		x^n\coloneqq x+\frac1n \bar{x},\quad n\in\mathbb{N}.
		\]
		For each $n\in\mathbb{N}$, note that $x^n\in x+\X_+^\sharp$, which implies that
		\[
		\ip{y^\ast,y}\geq \varphi_{G,y^\ast}(x)>\varphi_{G,y^\ast}(x^n)
		\]
		since $\varphi_{G,y^\ast}$ is strictly decreasing. Hence, $x^n\in \tilde{A}_y(y^\ast)$ for each $n\in\mathbb{N}$. Moreover, $(x^n)_{n\in\mathbb{N}}$ converges to $x$. Therefore, $x\in\cl(\tilde{A}_y(y^\ast))$.
		
		In each case, we establish $A_y(y^\ast)=\cl(\tilde{A}_y(y^\ast))$.
		
		\item Note that
		\begin{align*}
			\sup_{\substack{x\in\X,\\ y\in\Y}}f(x,y,y^\ast)&=\sup_{y\in\Y}\of{\sup_{x\in\X}\of{\ip{x^\ast,x}-I_{A(y^\ast)}(x,y)}-\varphi_{F,z^\ast}(y)}\\
			&=\sup_{y\in\Y}\of{\sup_{x\in\X}\of{\ip{x^\ast,x}-I_{A_y(y^\ast)}(x)}-\varphi_{F,z^\ast}(y)}\\
			&=\sup_{y\in\Y}\of{I^\ast_{A_y(y^\ast)}(x^\ast)-\varphi_{F,z^\ast}(y)}.
		\end{align*}
		Let us fix $y\in\Y$. Note that $I^\ast_{A_y(y^\ast)}(x^\ast)=\sup_{x\in A_y(y^\ast)}\ip{x^\ast,x}=-\sigma_{A_y(y^\ast)}(-x^\ast)$. Since $A_y(y^\ast)=\cl(\tilde{A}_y(y^\ast))$ by (i), the sets $A_y(y^\ast)$ and $\tilde{A}_y(y^\ast)$ have the same support function so that $I^\ast_{A_y(y^\ast)}(x^\ast)=I^\ast_{\tilde{A}_y(y^\ast)}(x^\ast)$. Therefore,
		\begin{align*}
			\sup_{\substack{x\in\X,\\ y\in\Y}}f(x,y,y^\ast)&=\sup_{y\in\Y}\of{I^\ast_{\tilde{A}_y(y^\ast)}(x^\ast)-\varphi_{F,z^\ast}(y)}\\
			&=\sup_{y\in\Y}\of{\sup_{x\in\X}\of{\ip{x^\ast,x}-I_{\tilde{A}_y(y^\ast)}(x)}-\varphi_{F,z^\ast}(y)}\\
			&=\sup_{y\in\Y}\of{\sup_{x\in\X}\of{\ip{x^\ast,x}-I_{\tilde{A}(y^\ast)}(x,y)}-\varphi_{F,z^\ast}(y)}\\
			&=\sup_{\substack{x\in\X,\\ y\in\Y}}\tilde{f}(x,y,y^\ast),
		\end{align*}
		which concludes the proof.
	\end{enumerate}
\end{proof}

We also recall the statement of Liu's minimax inequality.

\begin{theorem}\label{thm:liu}
	\citep{liu} Let $\W,\mathcal{V}$ be topological linear spaces and let $A\subseteq \W$, $B\subseteq\mathcal{V}$ be nonempty convex sets. Let $g,\tilde{g}\colon A\times B\to[-\infty,+\infty]$ be two functions satisfying the following properties:
	\begin{enumerate}[(i)]
		\item For each $v\in\mathcal{V}$, the function $w\mapsto g(w,v)$ is upper semicontinuous.
		\item For each $v\in\mathcal{V}$, the function $w\mapsto \tilde{g}(w,v)$ is quasiconcave.
		\item For each $w\in\mathcal{W}$, the function $v\mapsto g(w,v)$ is quasiconvex.
		\item For each $w\in\mathcal{W}$, the function $v\mapsto g(w,v)$ is lower semicontinuous.
		\item For each $(w,v)\in\W\times\mathcal{V}$, it holds $\tilde{g}(w,v)\leq g(w,v)$.
		\item $A$ is a compact set.
	\end{enumerate}
	Then, we have
	\[
	\inf_{v\in\mathcal{V}}\sup_{w\in\W}\tilde{g}(w,v)\leq \sup_{w\in\W}\inf_{v\in\mathcal{V}}g(w,v).
	\]
\end{theorem}

\begin{proof}\hspace{-0.3em}\textbf{of Theorem~\ref{thm:conj}}
	Let us fix $z^\ast\in \Z^+_+\setminus\{0\}$ such that $\varphi^\ast_{F\circ G,z^\ast}$ is a proper function. Let $x^\ast\in\X^\ast$. For each $x\in\X$ and $y\in\Y$, since $G(x)\in\G_+(\Y)$, by \eqref{separation} and the positive homogeneity of support functions, we have
	\begin{align}\label{eq:sep2}
		y\in G(x)\quad &\Leftrightarrow\quad \forall y^\ast\in\Y^+_+\setminus\{0\}\colon \ip{y^\ast,y}\geq \varphi_{G,y^\ast}(x)\notag \\
		&\Leftrightarrow\quad \forall y^\ast\in\mathcal{B}_{\Y^\ast}\colon \ip{y^\ast,y}\geq \varphi_{G,y^\ast}(x).
	\end{align}
	Using the definition of conjugate function and Proposition~\ref{prop:comp-sc}, we obtain
	\begin{align*}
		\varphi^\ast_{F\circ G,z^\ast}(x^\ast)&=\sup_{x\in\X}\of{\ip{x^\ast,x}-\varphi_{F\circ G,z^\ast}(x)}\\
		&=\sup_{x\in\X}\of{\ip{x^\ast,x}-\inf_{y\in G(x)}\varphi_{F,z^\ast}(y)}\\
		&=\sup_{\substack{x\in\X,\\ y\in G(x)}}\of{\ip{x^\ast,x}-\varphi_{F,z^\ast}(y)}.
	\end{align*}
	Combining this with \eqref{eq:sep2} gives
	\begin{equation}\label{eq:supinf}
		\varphi^\ast_{F\circ G,z^\ast}(x^\ast)=\sup_{\substack{x\in\X,\\ y\in\Y}}\cb{\ip{x^\ast,x}-\varphi_{F,z^\ast}(y)\mid \forall y^\ast\in\mathcal{B}_{\Y^\ast}\colon \ip{y^\ast,y}\geq \varphi_{G,y^\ast}(x)}.
	\end{equation}
	For each $y^\ast\in\mathcal{B}_{\Y^\ast}$, let us define $A(y^\ast), \tilde{A}(y^\ast)$ by \eqref{eq:A}. Using indicator functions and the property \eqref{indicatorcalculus}, we may rewrite \eqref{eq:supinf} as
	\begin{align*}
		\varphi^\ast_{F\circ G,z^\ast}(x^\ast)&=\sup_{\substack{x\in\X,\\ y\in\Y}}\cb{\ip{x^\ast,x}-\varphi_{F,z^\ast}(y)\mid \forall y^\ast\in\mathcal{B}_{\Y^\ast}\colon (x,y)\in A(y^\ast)}\\
		&=\sup_{\substack{x\in\X,\\ y\in\Y}}\of{\ip{x^\ast,x}-\varphi_{F,z^\ast}(y)-I_{
				\bigcap_{y^\ast\in\mathcal{B}_{\Y^\ast}} A(y^\ast)}(x,y)}\\
		&=\sup_{\substack{x\in\X,\\ y\in\Y}}\of{\ip{x^\ast,x}-\varphi_{F,z^\ast}(y)-\sup_{y^\ast\in\mathcal{B}_{\Y^\ast}} I_{A(y^\ast)}(x,y)}\\
		&=\sup_{\substack{x\in\X,\\ y\in\Y}}\inf_{y^\ast\in\mathcal{B}_{\Y^\ast}}\of{\ip{x^\ast,x}-\varphi_{F,z^\ast}(y)- I_{A(y^\ast)}(x,y)}\\
		&=\sup_{\substack{x\in\X,\\ y\in\Y}}\inf_{y^\ast\in\mathcal{B}_{\Y^\ast}}f(x,y,y^\ast).
	\end{align*}
	To be able to change the order of supremum and infimum in the last line, we use Liu's minimax inequality. Since $f(x,y,y^\ast)\geq \tilde{f}(x,y,y^\ast)$ for each $(x,y,y^\ast)\in\X\times\Y\times\mathcal{B}_{\Y^\ast}$, by Lemma~\ref{lem:proof1}, Theorem~\ref{thm:liu}, and Lemma~\ref{lem:proof2}(ii), we get
	\[
	\varphi^\ast_{F\circ G,z^\ast}(x^\ast)=\sup_{\substack{x\in\X,\\ y\in\Y}}\inf_{y^\ast\in\mathcal{B}_{\Y^\ast}}f(x,y,y^\ast)\geq \inf_{y^\ast\in\mathcal{B}_{\Y^\ast}}\sup_{\substack{x\in\X,\\ y\in\Y}}\tilde{f}(x,y,y^\ast)=\inf_{y^\ast\in\mathcal{B}_{\Y^\ast}}\sup_{\substack{x\in\X,\\ y\in\Y}}f(x,y,y^\ast).
	\]
	On the other hand, we have
	\[
	\sup_{\substack{x\in\X,\\ y\in\Y}}\inf_{y^\ast\in\mathcal{B}_{\Y^\ast}}f(x,y,y^\ast)\leq \inf_{y^\ast\in\mathcal{B}_{\Y^\ast}}\sup_{\substack{x\in\X,\\ y\in\Y}}f(x,y,y^\ast)
	\]
	by weak duality. Hence, we obtain
	\[
	\varphi^\ast_{F\circ G,z^\ast}(x^\ast)=\inf_{y^\ast\in\mathcal{B}_{\Y^\ast}}\sup_{\substack{x\in\X,\\ y\in\Y}}f(x,y,y^\ast)
	=\inf_{y^\ast\in\mathcal{B}_{\Y^\ast}}\sup_{\substack{x\in\X,\\ y\in\Y}}\of{\ip{x^\ast,x}-\varphi_{F,z^\ast}(y)- I_{A(y^\ast)}(x,y)}
	=\inf_{y^\ast\in\mathcal{B}_{\Y^\ast}}h(y^\ast),
	\]
	where
	\[
	h(y^\ast)\coloneqq\sup_{\substack{x\in\X,\\ y\in\Y}}\cb{\ip{x^\ast,x}-\varphi_{F,z^\ast}(y)\mid \ip{y^\ast,y}\geq \varphi_{G,y^\ast}(x)}
	\]
	for each $y^\ast\in\mathcal{B}_{\Y^\ast}$. We use Lagrange duality to calculate $h$. Let us fix $y^\ast\in\mathcal{B}_{\Y^\ast}$. Note that $(x,y)\mapsto \ip{x^\ast,x}-\varphi_{F,z^\ast}(y)$ is a concave upper semicontinuous function on $\X\times\Y$ as argued in the proof of Lemma~\ref{lem:proof1}(i). Similarly, $(x,y)\mapsto \varphi_{G,y^\ast}(x)-\ip{y^\ast,y}$ is a convex lower semicontinuous function on $\X\times\Y$.
	
	Since $\varphi^\ast_{F\circ G,z^\ast}$ is assumed to be a proper function, we have $h(y^\ast)>-\infty$ for each $y^\ast\in\mathcal{B}_{\Y^\ast}$ and also that $\varphi_{F\circ G,z^\ast}$ is proper. Then, by Corollary~\ref{cor:comp-sc}, $\varphi_{F,z^\ast}$ is proper.
	
	Note that $\varphi_{G,y^\ast}$ is a lower semicontinuous convex function on $\X$. Hence, there are three possibilities concerning the properness of this function:
	
	\textbf{Case 1:} Suppose that $\varphi_{G,y^\ast}$ is a proper function. Since $h(y^\ast)>-\infty$ and $\varphi_{F,z^\ast}$ is a proper function, there exists $(\tilde{x},\tilde{y})\in\dom(\varphi_{G,y^\ast})\times\dom(\varphi_{F,z^\ast})$ such that $\varphi_{G,y^\ast}(\tilde{x})\leq \ip{y^\ast,\tilde{y}}$. Moreover, Slater's condition also holds for this problem:
	\[
	\exists (x,y)\in\dom(\varphi_{G,y^\ast})\times\dom(\varphi_{F,z^\ast})\colon \varphi_{G,y^\ast}(x)<\ip{y^\ast,y}.
	\]
	Indeed, this condition holds trivially under Assumption~\ref{asmp:sc}(a). Next, suppose that Assumption~\ref{asmp:sc}(b) holds. Let $\bar{x}\in\X^\sharp_+$. Then, $\varphi_{G,y^\ast}(\tilde{x}+\bar{x})<\varphi_{G,y^\ast}(\tilde{x})\leq \ip{y^\ast,\tilde{y}}$. Hence, choosing $x=\tilde{x}+\bar{x}$ and $y=\tilde{y}$ verifies Slater's condition. Therefore, by strong duality theorem for convex optimization (see, e.g., \citet[Theorem~2.9.2]{zalinescu}), we have
	\[
	h(y^\ast)
	=\inf_{\lambda\geq 0}\sup_{\substack{x\in\X,\\ y\in\Y}}\of{\ip{x^\ast,x}-\varphi_{F,z^\ast}(y)+\lambda\ip{y^\ast,y}-\lambda \varphi_{G,y^\ast}(x)}.
	\]
	
	\textbf{Case 2:} Suppose that $\varphi_{G,y^\ast}(x)=+\infty$ for every $x\in\X$. In this case, the maximization problem for $h(y^\ast)$ has empty feasible region so that $h(y^\ast)=-\infty$, which is a contradiction. Hence, this case is not possible.
	
	\textbf{Case 3:} Suppose that $\varphi_{G,y^\ast}(\bar{x})=-\infty$ for some $\bar{x}\in\X$. Then, by \citet[Proposition~2.2.5]{zalinescu}, $\varphi_{G,y^\ast}(x)=-\infty$ for every $x\in \dom(\varphi_{G,y^\ast})$. Hence, the feasible region of the maximization problem for $h(y^\ast)$ is $\dom(\varphi_{G,y^\ast})\times\Y$. Then,
	\[
	h(y^\ast)= \sup_{\substack{x\in\dom(\varphi_{G,y^\ast}),\\ y\in\Y}}\of{\ip{x^\ast,x}-\varphi_{F,z^\ast}(y)}=
	\sup_{x\in\dom(\varphi_{G,y^\ast})}\ip{x^\ast,x} -\inf_{y\in\Y}\varphi_{F,z^\ast}(y).
	\]
	However, for every $\bar{y}^\ast\in\mathcal{B}_{\Y^\ast}$ such that $\varphi_{G,y^\ast}$ is proper (Case~1), we simply have $h(y^\ast)\geq h(\bar{y})$ since the feasible region of the maximization problem for $h(\bar{y}^\ast)$ is always a subset of $\dom(\varphi_{G,y^\ast})\times\Y$.
	
	Consequently, we may restrict our attention to the set of all $y^\ast\in\Y^+_+\setminus\{0\}$ for which $\varphi_{G,y^\ast}$ is proper, let us denote this set by $\Y^\ast_G$. It is easy to see that this set is a cone. Hence,
	\begin{align*}
		\varphi^\ast_{F\circ G,z^\ast}(x^\ast)&=\inf_{y^\ast\in\Y^\ast_G\cap \mathcal{B}_{\Y^\ast}}h(y^\ast)\\
		&=\inf_{y^\ast\in\Y^\ast_G\cap \mathcal{B}_{\Y^\ast}}\inf_{\lambda\geq 0}\sup_{\substack{x\in\X,\\ y\in\Y}}\of{\ip{x^\ast,x}-\varphi_{F,z^\ast}(y)+\lambda\ip{y^\ast,y}-\lambda \varphi_{G,y^\ast}(x)}\\
		&=\inf_{y^\ast\in\Y^\ast_G\cap \mathcal{B}_{\Y^\ast}}\inf_{\lambda\geq 0}\sup_{\substack{x\in\X,\\ y\in\Y}}\of{\ip{x^\ast,x}-\varphi_{F,z^\ast}(y)+\ip{\lambda y^\ast,y}- \varphi_{G,\lambda y^\ast}(x)}\\
		&=\inf_{y^\ast\in\Y^\ast_G}\sup_{\substack{x\in\X,\\ y\in\Y}}\of{\ip{x^\ast,x}-\varphi_{F,z^\ast}(y)+\ip{y^\ast,y}-\varphi_{G,y^\ast}(x)}\\
		&=\inf_{y^\ast\in\Y^\ast_G}
		\of{\sup_{x\in\X}\of{\ip{x^\ast,x}-\varphi_{G,y^\ast}(x)}+\sup_{y\in\Y}\of{\ip{y^\ast,y}-\varphi_{F,z^\ast}(y)}}\\
		&=\inf_{y^\ast\in\Y^\ast_G}\of{\varphi^\ast_{G,y^\ast}(x)+\varphi^\ast_{F,z^\ast}(y)}.
	\end{align*}
	Finally, the formula for $-(F\circ G)^\ast$ follows immediately by Remark~\ref{rem:SVconjugate}(ii).
\end{proof}

\section{Conclusion}\label{sec:conc}

In this paper, we prove a formula for the conjugate of the composition of two set-valued functions taking values in a complete lattice. Combined with the set-valued biconjugation theorem, it yields a dual representation for the composition when the composition is guaranteed to be lattice-lower semicontinuous and proper. Due to the technical nature of the proof, we limit the scope of this paper to theoretical results. As a future direction, the consequences of this formula for set-valued convex risk measures can be studied.

\bibliographystyle{named}

\end{document}